\documentclass[11pt]{article}

\usepackage{enumerate}
\usepackage{amsmath}
\usepackage{amssymb,latexsym}
\usepackage{amsthm}
\usepackage{color}
\usepackage{cancel}
\usepackage{graphicx}

\usepackage{hyperref}

\usepackage{amscd}

\DeclareMathOperator{\C}{\mathcal{C}}
\DeclareMathOperator{\V}{\mathbb{V}}

\newtheorem{theorem}{Theorem}[section]

\newtheorem{lemma}[theorem]{Lemma}
\newtheorem{corollary}[theorem]{Corollary}
\newtheorem{definition}[theorem]{Definition}
\newtheorem{proposition}[theorem]{Proposition}

\newtheorem{example}[theorem]{Example}

\newtheorem{remark}[theorem]{Remark}

\newcommand{\fqn}{\mathbb{F}_{q^n}}

\newcommand{\cC}{{\mathcal C}}
\newcommand{\cS}{{\mathcal S}}

\newcommand{\F}{{\mathbb F}}

\newcommand{\w}{{\mathbf w}}

\newcommand{\fq}{{\mathbb F}_{q}}

\newcommand{\la}{\langle}
\newcommand{\ra}{\rangle}

\newcommand{\PG}{\mathrm{PG}}

\title{Scattered subspaces and related codes}

\author{Giovanni Zini and Ferdinando Zullo\thanks{
The first author is funded by the project ''Attrazione e Mobilità dei
Ricercatori'' Italian PON Programme (PON-AIM 2018 num. AIM1878214-2).
The research was supported by the project ''VALERE: VAnviteLli pEr la RicErca" of the University of Campania ''Luigi Vanvitelli'', and by the Italian National Group for Algebraic and Geometric Structures and their Applications (GNSAGA - INdAM).}}
\date{}

\begin{document}
\maketitle

\begin{center}
\emph{Dedicated to the memory of Elisa Montanucci.\\ We unite us to her family's pain.}
\end{center}

\begin{abstract}
After a seminal paper by Shekeey (2016), a connection between maximum $h$-scattered $\fq$-subspaces of $V(r,q^n)$ and maximum rank distance (MRD) codes has been established in the extremal cases $h=1$ and $h=r-1$.
In this paper, we propose a connection for any $h\in\{1,\ldots,r-1\}$, extending and unifying all the previously known ones.
As a consequence, we obtain examples of non-square MRD codes which are not equivalent to generalized Gabidulin or twisted Gabidulin codes.
Up to equivalence, we classify MRD codes having the same parameters as the ones in our connection.
Also, we determine the weight distribution of codes related to the geometric counterpart of maximum $h$-scattered subspaces.
\end{abstract}



\section{Introduction}

An $\fq$-subspace $U$ of an $r$-dimensional $\fqn$-vector space $V$ is said to be \emph{$h$-scattered} if $U$ spans $V$ over $\fqn$ and, for any $h$-dimensional $\fqn$-subspace $H$ of $V$, $U$ meets $H$ in an $\fq$-subspace of dimension at most $h$.
This family of subspaces was introduced in \cite{CsMPZ} as a generalization of $1$-scattered subspaces, which are simply known as scattered subspaces and were originally presented in \cite{BL2000}. Since then, the theory of scattered subspaces has constantly increased its importance, mainly because of their applications to several algebraic and geometric objects, such as finite semifields, blocking sets, two-intersection sets; see \cite{Lavrauw, LVdV2015,Polverino}. After the seminal paper \cite{Sheekey} by Sheekey, the interest towards scattered subspaces was also boosted by their connections with the theory of rank metric codes, whose relevance in communication theory relies on its applications to random linear network coding and cryptography.

A $h$-scattered $\fq$-subspace of highest dimension in $V(r,q^n)$ is called \emph{maximum $h$-scattered}; its dimension is upper bounded by $\frac{rn}{h+1}$.
This bound is known to be achieved in the following cases: $h=1$, $h=r-1$, $(h+1)\mid r$ or $h=n-3$; see Section \ref{sec:h-scatt}.
When $h=1$ or $h=r-1$, maximum $h$-scattered subspaces are strongly related to rank metric codes having the greatest correcting and detecting capabilities for fixed dimension and ambient space, that is, to maximum rank distance (MRD) codes. This has been shown in \cite{Sheekey,CSMPZ2016,PZ} for $h=1$ and in \cite{Lunardon2017,ShVdV} for $h=r-1$, while no relation was known for $1<h<r-1$.
In this paper we establish a connection between $\fq$-subspaces of $V$ and rank metric codes. We start by generalizing the construction of rank-metric codes $\mathcal{C}_U$ provided in \cite{CSMPZ2016} and defined by an $\fq$-subspace $U$ of $V$.
We detect those $U$'s such that $\mathcal{C}_U$ is MRD; among these are the maximum $1$- and $(r-1)$-scattered subspaces. Actually, the code $\mathcal{C}_U$ is MRD exactly when $U$ is the dual of a $h$-scattered subspace of dimension $\frac{rn}{h+1}$, for some $1\leq h\leq r-1$. Therefore, our connection extends and unifies the ones in \cite{Sheekey,CSMPZ2016,PZ,ShVdV,Lunardon2017}.
To this aim, we exhibit two characterizations of $h$-scattered subspaces of dimension $\frac{rn}{h+1}$, which are of independent interest.
Moreover we prove that, up to equivalence, the MRD codes of type $\mathcal{C}_U$ are exactly the $\fq$-linear MRD codes with parameters $(\frac{rn}{h+1},n,q;n-h)$ and maximum right idealiser.

An essential though difficult task is to decide whether or not two rank metric codes with the same parameters are equivalent (especially when they correspond to non-square matrices).
A remarkable aspect of the MRD codes that we construct is that we are able to determine one of their idealisers; this allows to prove that some of them are not equivalent to punctured generalized Gabidulin codes nor to punctured generalized twisted Gabidulin codes.

The geometric counterparts of $h$-scattered subspaces of dimension $\frac{rn}{h+1}$ are called \emph{$h$-scattered linear sets} of rank $\frac{rn}{h+1}$.
They are known to have at most $h+1$ intersection numbers with respect to the hyperplanes, and hence are of interest in coding theory when regarded as projective systems.
The intersection numbers w.r.t.\ the hyperplanes of $h$-scattered linear sets of rank $\frac{rn}{h+1}$ have been determined in \cite{BL2000} for $h=1$, in \cite{NZ} for $h=2$, and in \cite{ShVdV} for $h=r-1$.
We determine them for any $1\leq h\leq r-1$, by using the connection between MRD codes and $h$-scattered subspaces of dimension $\frac{rn}{h+1}$ presented in Section \ref{sec:subMRD}.
As a byproduct, we compute the weight distribution of the arising codes; this answers a question posed by Randrianarisoa \cite{Ra}.

The paper is organized as follows.
Section \ref{sec:pre} contains preliminary results on $h$-scattered subspaces (Section \ref{sec:h-scatt}), dualities of subspaces, both ordinary and Delsarte (Section \ref{sec:dualities}), linear codes, equipped with the Hamming distance or with the rank metric (Section \ref{sec:codes}).
In Section \ref{sec:subMRD} we describe the connection between $\fq$-subspaces and rank metric codes, characterizing those codes which are MRD, and showing that $\fq$-linear MRD $(\frac{rn}{h+1},n,q;n-h)$-codes with maximum right idealiser are exactly the codes of type $\mathcal{C}_U$, up to equivalence.
This connection is shown to extend and unify the previously known ones in Section \ref{sec:uni}.
Section \ref{sec:twocharact} completes the connection between $h$-scattered subspaces of dimension $\frac{rn}{h+1}$ and MRD codes, by means of two characterizations which are proved through the ordinary and Delsarte dualities.
Section \ref{sec:noGab} provides families of MRD codes which are not equivalent to punctured generalized (twisted) Gabidulin codes.
Section \ref{sec:h+1weights} computes the weight distribution of the linear codes arising from $h$-scattered linear sets of rank $\frac{rn}{h+1}$, seen as projective systems.
Finally, in Section \ref{sec:open}, we resume our results and state some open questions.

\section{Preliminaries}\label{sec:pre}

\subsection{Scattered $\mathbb{F}_q$-subspaces with respect to $\mathbb{F}_{q^n}$-subspaces}\label{sec:h-scatt}

Let $V=V(m,q)$ denote an $m$-dimensional $\F_q$-vector space.
A $t$-spread of $V$ is a set $\cS$ of $t$-dimensional $\F_q$-subspaces such that each vector of $V^*=V\setminus \{{\bf 0}\}$ is contained in exactly one element of $\cS$.
As shown by Segre in \cite{Segre}, a $t$-spread of $V$ exists if and only if $t$ divides $m$.

Let $V$ be an $r$-dimensional $\F_{q^n}$-vector space and let $\cS$ be an $n$-spread of $V$.
An $\F_q$-subspace $U$ of $V$ is called \emph{scattered} w.r.t.\ $\cS$ if $U$ meets every element of $\cS$ in an $\F_q$-subspace of dimension at most one; see \cite{BL2000}.
If we consider $V$ as an $rn$-dimensional $\F_q$-vector space, then it
is well-known that the one-dimensional $\F_{q^n}$-subspaces of $V$, viewed as $n$-dimensional $\F_q$-subspaces, form an $n$-spread of $V$. This spread is called the \emph{Desarguesian spread}.
In this paper scattered will always mean scattered w.r.t.\ the Desarguesian spread.
Blokhuis and Lavrauw \cite{BL2000} showed that the dimension of such subspaces is bounded by $rn/2$.
After a series of papers it is now known that when $rn$ is even there always exist scattered subspaces of dimension $rn/2$; they are called \emph{maximum scattered} \cite{BBL2000, BGMP2015, BL2000, CSMPZ2016}.

In \cite{CsMPZ}, the authors introduced a special family of scattered subspaces, named $h$-scattered subspaces.
Let $V$ be an $r$-dimensional $\F_{q^n}$-vector space and $h\leq r-1$ be a positive integer.
An $\F_q$-subspace $U$ of $V$ is called $h$-\emph{scattered} (or scattered w.r.t.\ the $h$-dimensional $\fqn$-subspaces) if $\la U \ra_{\F_{q^n}}=V$ and each $h$-dimensional $\F_{q^n}$-subspace of $V$ meets $U$ in an $\F_q$-subspace of dimension at most $h$.
The $1$-scattered subspaces are the scattered subspaces generating $V$ over $\F_{q^n}$. The same definition applied to $h=r$ describes the $n$-dimensional $\F_q$-subspaces of $V$ defining canonical subgeometries of $\PG(V,\F_{q^n})$. If $h=r-1$ and $\dim_{\F_q}(U)=n$, then $U$ is $h$-scattered exactly when $U$ defines a scattered $\F_q$-linear set with respect to the hyperplanes, introduced in \cite[Definition 14]{ShVdV}; see also \cite{Lunardon2017}.

Theorem \ref{th:bound} bounds the dimension of a $h$-scattered subspace.

\begin{theorem}\label{th:bound}{\rm \cite[Theorem 2.3]{CsMPZ}}
If $U$ is a $h$-scattered $\F_q$-subspace of dimension $k$ in $V=V(r,q^n)$, then one of the following holds:
\begin{itemize}
\item $k=r$ and $U$ defines a subgeometry $\PG(r-1,q)$ of $\PG(V,\F_{q^n})$;
\item $k\leq\frac{rn}{h+1}$.
\end{itemize}
\end{theorem}

A $h$-scattered $\F_q$-subspace of highest possible dimension is said to be a {\em maximum $h$-scattered} $\F_q$-subspace.
Theorem \ref{th:inter} bounds the dimension of the intersection between a $h$-scattered subspace of dimension $\frac{rn}{h+1}$ and an $\fqn$-subspace of codimension $1$.

\begin{theorem}{\rm \cite[Theorem 2.8]{CsMPZ}}
	\label{th:inter}
	If $U$ is an $\frac{rn}{h+1}$-dimensional $h$-scattered $\F_q$-subspace of a vector space $V=V(r,q^n)$, then for any $(r-1)$-dimensional $\F_{q^n}$-subspace $H$ of $V$ we have
	\[\frac{rn}{h+1}-n\leq \dim_{\F_q}(U \cap H) \leq \frac{rn}{h+1}-n+h.\]
\end{theorem}

Constructions of $h$-scattered $\F_q$-subspaces have been given in \cite{CsMPZ} and also in \cite{NPZZ}.
A generalization of $h$-scattered subspaces has been recently introduced in \cite{BCsMT}.

\subsection{Two dualities for $\mathbb{F}_q$-subspaces}\label{sec:dualities}

In this paper we need both ordinary and Delsarte dualities.

\subsubsection{Ordinary duality}\label{sec:classicalduality}

Let $\sigma \colon V\times V \rightarrow \mathbb{F}_{q^n}$ be a non-degenerate reflexive sesquilinear form over $V=V(r,q^n)$ and define
$\sigma' \colon V \times V \rightarrow \mathbb{F}_q, \, (\mathbf{u},\mathbf{v})\mapsto \mathrm{Tr}_{q^n/q}(\sigma(\mathbf{u},\mathbf{v}))$.
Once we regard $V$ as an $rn$-dimensional $\F_q$-vector space, $\sigma^\prime$ turns out to be a non-degenerate reflexive sesquilinear form over $V=V(rn,q)$.
Let $\perp$ and $\perp'$ be the orthogonal complement maps defined by $\sigma$ and $\sigma'$ on the lattices of the $\F_{q^n}$-subspaces and the $\F_q$-subspaces of $V$, respectively. The following properties hold (see \cite[Section 2]{Polverino} for the details).
\begin{itemize}
    \item $\dim_{\F_{q^n}}(W)+\dim_{\F_{q^n}}(W^\perp)=r$, for every $\F_{q^n}$-subspace $W$ of $V$.
    \item $\dim_{\F_{q}}(U)+\dim_{\F_{q}}(U^{\perp'})=nr$, for every $\F_{q}$-subspace $U$ of $V$.
    \item $W^\perp=W^{\perp'}$, for every $\F_{q^n}$-subspace $W$ of $V$.
    \item Let $W$ and $U$ be an $\F_{q^n}$-subspace and an $\F_q$-subspace of $V$ of dimension $s$ and $t$, repsectively. Then
    \begin{equation}\label{eq:dualweight} \dim_{\F_q}(U^{\perp'}\cap W^{\perp'})-\dim_{\F_q}(U\cap W)=rn-\dim_{\F_q}(U)-sn. \end{equation}
    \item Let $\sigma$, $\sigma_1$ be non-degenerate reflexive sesquilinear forms over $V$ and define $\sigma^\prime$, $\sigma_1^\prime$, $\perp$, $\perp_1$, $\perp'$ and $\perp_1'$ as above. Then there exists an invertible $\F_{q^n}$-linear map $f$ such that $f(U^{\perp'})=U^{\perp_1'}$, i.e. $U^{\perp'}$ and $U^{\perp_1'}$ are $\mathrm{GL}(V)$-equivalent. 
\end{itemize}
When $U$ is an $\F_q$-subspace of $V$, we denote by $U^{\perp_O}$ one of the $\F_q$-subspaces $U^{\perp'}$, where $\perp'$ is defined by the restriction to $\F_q$ of any non-degenerate reflexive sesquilinear form over $V$, as defined at the beginning of this section.

\subsubsection{Delsarte duality}\label{sec:Delsarteduality}

Let $U$ be a $k$-dimensional $\F_q$-subspace of a vector space $V=V(r,q^n)$, with $k>r$. By \cite[Theorems 1, 2]{LuPo2004} (see also \cite[Theorem 1]{LuPoPo2002}), there is an embedding of $V$ in $\V=V(k,q^n)$ with $\V=V \oplus \Gamma$ for some $(k-r)$-dimensional $\F_{q^n}$-subspace $\Gamma$ such that
$U=\la W,\Gamma\ra_{\F_{q}}\cap V$, where $W$ is a $k$-dimensional $\F_q$-subspace of $\V$ satisfying $\langle W\rangle_{\F_{q^n}}=\V$ and $W\cap \Gamma=\{{\bf 0}\}$.
Then $ \varphi:V\to\V/\Gamma$, $\mathbf{v}\mapsto \mathbf{v}+\Gamma$, is an $\fqn$-isomorphism such that $\varphi(U)=W+\Gamma$.

Following \cite[Section 3]{CsMPZ}, let $\beta'\colon W\times W\rightarrow\F_{q}$ be a non-degenerate reflexive sesquilinear form on $W$. Then $\beta'$ can be extended to a non-degenerate reflexive sesquilinear form $\beta\colon \V\times\V\rightarrow\F_{q^n}$. Let $\perp$ and $\perp'$ be the orthogonal complement maps defined by $\beta$ and $\beta'$ on the lattices of $\F_{q^n}$-subspaces of $\V$ and of $\F_q$-subspaces of $W$, respectively.
For an $\F_q$-subspace $S$ of $W$ the $\F_{q^n}$-subspace $\la S \ra_{\F_{q^n}}$ of $\V$ will be denoted by $S^*$. In this case, $(S^*)^{\perp}=(S^{\perp'})^*$.


\begin{definition}
\label{deffff}	
\rm
Let $U$ be a $k$-dimensional $\F_q$-subspace of $V=V(r,q^n)$ such that $k>r$ and $\dim_{\F_q}(M\cap U)<k-1$ for every $(r-1)$-dimensional $\fqn$-subspace $M$ of $V$. Then the $k$-dimensional $\F_q$-subspace $W+\Gamma^{\perp}$ of the quotient space $\V/\Gamma^{\perp}$  will be denoted by $ U^{\perp_{D}}$ and will be called the \emph{Delsarte dual} of $U$ (w.r.t.\ $\perp$).
\end{definition}

The Delsarte duality preserves the property of being scattered w.r.t.\ $\fqn$-subspaces, in the following sense.

\begin{theorem}\cite[Theorem 3.3]{CsMPZ}
\label{thm:dual}
Let $U$ be a $k$-dimensional $h$-scattered $\F_q$-subspace of a vector space $V=V(r,q^n)$ with $n\geq h+3$.  Then $U^{\perp_{D}}$ is an $\frac{rn}{h+1}$-dimensional $(n-h-2)$-scattered $\fq$-subspace of $\V/\Gamma^\perp=V(k-r,q^n)$.
\end{theorem}

Proposition \ref{prop:property} points out some properties of the Delsarte duality.

\begin{proposition}\label{prop:property}
Let $U$, $W$, $V$, $\Gamma$, $\V$, $\perp$ and $\perp_D$ be defined as above.
The following properties hold:
\begin{itemize}
    \item $(U^{\perp_D})^{\perp_D}=W+\Gamma=\varphi^{-1}(U)$;
    \item under the assumption $n\geq h+3$, $U$ is an $\frac{rn}{h+1}$-dimensional $h$-scattered $\fq$-subspace of $V$ if and only if $U^{\perp_D}$ is an $\frac{rn}{h+1}$-dimensional $(n-h-2)$-scattered $\fq$-subspace of $\V/\Gamma^{\perp}$.
\end{itemize}
\end{proposition}

\begin{proof}
The first property easily follows from the definition of Delsarte duality.
Together with Theorem \ref{thm:dual} applied to $U^{\perp_D}$, this yields the second property.
\end{proof}

\subsection{Generalities on codes}\label{sec:codes}

In this section we recall some properties of codes that will be used in the paper. In Section \ref{sec:Hamming} we consider $\fq$-linear codes with respect to the Hamming metric in $\F_{q}^N$, while in Section \ref{sec:rank} we consider $\fq$-linear codes with respect to the rank metric in $\F_{q}^{m\times n}$.

\subsubsection{Projective systems and linear codes}\label{sec:Hamming}

Let $\mathcal{C}\subseteq\F_q^N$ be an $\fq$-linear code of length $N$, dimension $k$ and minimum distance $d$ over the alphabet $\fq$; we denote by $[N,k,d]_q$ the parameters of $\mathcal{C}$.
A generator matrix of $\mathcal{C}$ is a matrix $G\in\F_q^{k\times N}$ whose rows form a basis of $\mathcal{C}$.
The weight of a codeword $\mathbf{c}\in\mathcal{C}$ is the number of nonzero components of $\mathbf{c}$, and $A_i^H$ will denote the number of codewords of weight $i$ in $\mathcal{C}$.
The $N$-tuple $(A_0^H=1,A_1^H,\ldots,A_N^H)$ is called the weight distribution of $\mathcal{C}$, and the polynomial $\sum_{i=0}^{N}A_i^H z^i$ is the weight enumerator of $\mathcal{C}$.

A projective $[N,k,d]_q$-system is a point subset $\mathcal{P}$ of $\Omega=\PG(k-1,q)$ of size $N$, not contained in any hyperplane of $\Omega$, such that
\[d= N - \max\{|\mathcal{P}\cap\mathcal{H}|\colon \mathcal{H}\,\mbox{ is a hyperlane of }\,\Omega\}.  \]
The matrix $G\in\F_q^{k\times N}$ whose columns are the coordinates of the points of a projective $[N,k,d]_q$-system $\mathcal{P}$ is the generator matrix of a linear code with parameters $[N,k,d]_q$.
Different choices of the coordinates yield linear codes which are equivalent by means of a diagonal matrix; we denote one of them by $\mathcal{C}_{\mathcal{P}}$.

\begin{proposition}\label{prop:projsyst}
Let $\mathcal{P}$ be a projective $[N,k,d]_q$-system of $\Omega$ and $\mathcal{C}_{\mathcal{P}}$ be a corresponding linear $[N,k,d]_q$-code.
Then the weights of $\mathcal{C}_{\mathcal{P}}$ are the values $N-i$, where $i=|\mathcal{P}\cap\mathcal{H}|$ and $\mathcal{H}$ runs over the hyperplanes of $\Omega$.
The number $A_i^H$ of codewords of $\mathcal{C}_{\mathcal{P}}$ with weight $i$ is equal to the number of hyperplanes $\mathcal{H}$ of $\Omega$ such that $|\mathcal{P}\cap\mathcal{H}|=i$. 
\end{proposition}

\subsubsection{Rank metric codes}\label{sec:rank}

Rank metric codes were introduced by Delsarte \cite{Delsarte} in 1978 and they have been intensively investigated in recent years because of their applications; we refer to \cite{sheekey_newest_preprint} for a survey on this topic.
The set $\fq^{m\times n}$ of $m \times n$ matrices over $\fq$ may be endowed with a metric, called \emph{rank metric}, defined by
\[d(A,B) = \mathrm{rk}\,(A-B).\]
A subset $\C \subseteq \fq^{m\times n}$ equipped with the rank metric is called a \emph{rank metric code} (shortly, an \emph{RM code}).
The minimum distance of $\C$ is defined as
\[d = \min\{ d(A,B) \colon A,B \in \C,\,\, A\neq B \}.\]
Denote the parameters of an RM code $\C\subseteq\fq^{m\times n}$ with minimum distance $d$ by $(m,n,q;d)$.
We are interested in $\fq$-\emph{linear} RM codes, i.e. $\fq$-subspaces of $\fq^{m\times n}$.
Delsarte showed in \cite{Delsarte} that the parameters of these codes must obey a Singleton-like bound.

\begin{theorem}\label{th:Singleton}
If $\C$ is an RM code of $\F_q^{m\times n}$ with minimum distance $d$, then
\[ |\C| \leq q^{\max\{m,n\}(\min\{m,n\}-d+1)}. \]
\end{theorem}

When equality holds, we call $\C$ a \emph{maximum rank distance}\index{maximum rank distance} (\emph{MRD} for short) code.
Examples of MRD codes are resumed in \cite{PZ,sheekey_newest_preprint}, see also the paper \cite{SheekeyLondon}.

For an RM code $\cC\subseteq \F_{q}^{m \times n}$, the \emph{adjoint code} of $\C$ is
\[ \C^\top =\{C^t \colon C \in \C\}, \]
where $C^t$ is the transpose matrix of $C$.
Define the symmetric bilinear form $\langle\cdot,\cdot\rangle$ on $\F_q^{m \times n}$ by
\[ \langle M,N \rangle= \mathrm{Tr}(MN^t). \]
The \emph{Delsarte dual code} of an $\F_q$-linear RM code $\C\subseteq \F_{q}^{m \times n}$ is
\[ \C^\perp = \{ N \in \F_q^{m\times n} \colon \langle M,N \rangle=0 \; \text{for each} \; M \in \C \}. \]

\begin{remark}\label{rk:dualMRD}
If $\mathcal{C}\subseteq \F_{q}^{m \times n}$ is an MRD code with minimum distance $d$, then $\C^\top$ and $\mathcal{C^\perp}$ are MRD codes with minimum distances $d$ and $\min\{m,n\}-d+2$, respectively; see \cite{Delsarte,Ravagnani}.
\end{remark}

Given an RM code $\mathcal{C}$ in $\mathbb{F}_{q}^{m\times n}$ and an integer $i \in \mathbb{N}$, define $A_i=|\{M \in \mathcal{C} \colon \mathrm{rk}(M)=i\}|$. The \emph{rank distribution} of $\mathcal{C}$ is the vector $(A_i)_{i \in \mathbb{N}}$.
MacWilliams identities for RM codes are stated in Theorem \ref{th:MacWilliams} and were first obtained by Delsarte in \cite{Delsarte} using the machinery of association schemes; see also \cite{Ravagnani} for a different approach.
Recall that the $q$-binomial coefficient of two integers $s$ and $t$ is
\[  {s \brack t}_q=\left\{ \begin{array}{lll} 0 & \text{if}\,\, s<0,\,\,\text{or}\,\,t<0,\,\, \text{or}\,\, t>s,\\ 
1 & \text{if}\,\, t=0\,\, \text{and}\,\, s\geq 0,\\
\displaystyle\prod_{i=1}^t \frac{q^{s-i+1}-1}{q^i-1} & \text{otherwise}. \end{array} \right.  \]
\begin{theorem}(\cite[Theorem 3.3]{Delsarte},\cite[Theorem 31]{Ravagnani})\label{th:MacWilliams}
Let $\mathcal{C}$ be an RM code in $\mathbb{F}_q^{m\times n}$. Let $(A_i)_{i\in \mathbb{N}}$ and $(B_j)_{j\in \mathbb{N}}$ be the rank distribution of $\mathcal{C}$ and $\mathcal{C}^\perp$, respectively. For any integer $\nu \in \{ 0,\ldots,m \}$ we have
\[ \sum_{i=0}^{m-\nu} A_i  {m-i \brack \nu}_q = \frac{|\mathcal{C}|}{q^{n\nu}} \sum_{j=0}^\nu B_j {m-j \brack \nu -j}_q.\]
\end{theorem}

As a consequence, Delsarte in \cite{Delsarte} and later Gabidulin in \cite{Gabidulin} determined precisely the weight distribution of MRD codes.

\begin{theorem}\label{th:weightdistribution}
Let $\mathcal{C}$ be an MRD code in $\mathbb{F}_q^{m\times n}$ with minimum distance $d$. Let $m'=\min\{m,n\}$ and $n'=\max\{m,n\}$.
Then
\[ A_{d+\ell}={m'\brack d+\ell}_q \sum_{t=0}^\ell (-1)^{t-\ell}{\ell+d \brack \ell-t}_q q^{\binom{\ell-t}{2}}(q^{n'(t+1)}-1) \]
for any $\ell \in \{0,1,\ldots,n'-d\}$.
\end{theorem}

In particular, Lemma \ref{lemma:weight} holds.

\begin{lemma}(\cite[Lemma 2.1]{LTZ2},\label{lemma:weight}\cite[Lemma 52]{Ravagnani})\label{lemma:complete weight}
  Let $\mathcal{C}$ be an MRD code in $\mathbb{F}_q^{m\times n}$ with minimum distance $d$. Let $m'=\min\{m,n\}$ and $n'=\max\{m,n\}$. 
  Assume that the null matrix $O$ is in $\mathcal{C}$.
  Then, for any $0 \leq \ell \leq m'-d$, we have $A_{d+\ell}>0$, i.e. there
  exists at least one matrix $C \in \mathcal{C}$ such that $\mathrm{rk} (C) = d + \ell$.
\end{lemma}

Theorem \ref{th:dualrelations} follows from the MacWilliam identities.

\begin{theorem}\label{th:dualrelations}(\cite[Proof of Corollary 44]{Ravagnani})
Let $\mathcal{C}$ be an MRD code in $\mathbb{F}_q^{m\times n}$ with minimum distance $d$. Let $m'=\min\{m,n\}$ and $n'=\max\{m,n\}$. Then for any $\nu \in \{0,\ldots,m'-d \}$ we have
\begin{equation}\label{eq:identities} {m' \brack \nu}_q+\sum_{i=d}^{m'-\nu} A_i {m'-i \brack \nu}_q=\frac{|\mathcal{C}|}{q^{n'\nu}} {m' \brack \nu}_q. 
\end{equation}
\end{theorem}
\begin{proof}
By Remark \ref{rk:dualMRD}, the minimum distance of $\mathcal{C}^\perp$ is $m'-d+2$. 
Thus, Theorem \ref{th:MacWilliams} proves the claim.
\end{proof}

Two RM codes $\C$ and $\C'$ in $\mathbb{F}_q^{m\times n}$ are \emph{equivalent} if and only if there exist $X \in \mathrm{GL}(m,q)$, $Y \in \mathrm{GL}(n,q)$, $Z \in \F_q^{m\times n}$ and a field automorphism $\sigma$ of $\F_q$ such that
\[\C'=\{XC^\sigma Y + Z \colon C \in \C\}.\]
The \emph{left} and \emph{right idealisers}
$L(\C)$ and $R(\C)$ of an RM code $\mathcal{C}\subseteq\F_{q}^{m\times n}$ are defined as
\[ L(\C)=\{ Y \in \F_q^{m \times m} \colon YC\in \C\hspace{0.1cm} \text{for all}\hspace{0.1cm} C \in \C\},\]
\[ R(\C)=\{ Z \in \F_q^{n \times n} \colon CZ\in \C\hspace{0.1cm} \text{for all}\hspace{0.1cm} C \in \C\}.\]
The notion of idealisers have been introduced by  Liebhold and Nebe in \cite[Definition 3.1]{LN2016}; they are invariant under equivalences of rank metric codes. Further invariants have been introduced in \cite{GiuZ,NPH2}.
In \cite{LTZ2}, idealisers have been studied in details and the following result has been proved.

\begin{theorem}\label{th:propertiesideal}
Let $\mathcal{C}$ and $\mathcal{C}^\prime$ be $\fq$-linear RM codes of $\fq^{m\times n}$.
\begin{itemize}
    \item If $\mathcal{C}$ and $\mathcal{C}^\prime$ are equivalent, then their left and right idealisers are isomorphic as $\fq$-algebras (\cite[Proposition 4.1]{LTZ2}).
    \item $L(\C^\top)=R(\C)^\top$ and $R(\C^\top)=L(\C)^\top$ (\cite[Proposition 4.2]{LTZ2}).
    \item Let $\mathcal{C}$ have minimum distance $d>1$.
    If $m \leq n$, then $L(\C)$ is a finite field with $|L(\C)|\leq q^m$.
If $m \geq n$, then $R(\C)$ is a finite field with $|R(\C)|\leq q^n$.
In particular, when $m=n$, $L(\C)$ and $R(\C)$ are both finite fields (\cite[Theorem 5.4 and Corollary 5.6]{LTZ2}).
\end{itemize}
\end{theorem}




Let $\mathcal{C}$ be an RM code in $\fq^{n\times n}$, and $A\in\fq^{m\times n}$ be a matrix of rank $m\leq n$.
The RM code $A\mathcal{C}=\{AM\colon M\in\mathcal{C}\}\subseteq\fq^{m\times n}$ is a \emph{punctured code} obtained by \emph{puncturing $\mathcal{C}$ with $A$}. 

\begin{theorem}\label{th:punct}(\cite[Corollary 35]{BR}, \cite[Theorem 3.2]{CsS})
Let $\C$ be an MRD code with parameters $(n,n,q;d)$, $A \in \F_q^{m\times n}$ be a matrix of rank $m$, and $n-d\leq m\leq n$. Then the punctured code $A\C$ is an MRD code with parameters $(m,n,q;d+m-n)$ and $(A\C)^\top$ is an MRD code with parameters $(n,m,q;d+m-n)$.
\end{theorem}

In the literature equivalent representations of RM codes are used, other than the matrix representation that has been described above, and some of them will be used in this paper.
In particular, we see the elements of an $\fq$-linear RM code $\mathcal{C}$ with parameters $(m,n,q;d)$ as:
\begin{itemize}
    \item matrices of $\F_q^{m\times n}$ having rank at least $d$;
    \item $\fq$-linear maps $V\to W$ where $V=V(n,q)$ and $W=V(m,q)$, having usual map rank at least $d$;
    \item when $m=n$, elements of the $\fq$-algebra $\mathcal{L}_{n,q}$ of $q$-polynomials over $\fqn$ modulo $x^{q^n}-x$, having rank at least $d$ as an $\fq$-linear map $\fqn\to\fqn$.
\end{itemize}

\section{Connection between $\fq$-vector spaces and rank metric codes}\label{sec:subMRD}

In this section, an $\fq$-linear RM code with parameters $(m,n,q;d)$ is regarded as a set of $\fq$-linear maps $W_1=V(n,q)\to W_2=V(m,q)$. The following notation will be used.
\begin{itemize}
    \item $\omega_{\alpha}:\fqn\to\fqn$, $x\mapsto\alpha x$, for any $\alpha\in\fqn$.
    \item $\mathcal{F}_n=\{\omega_{\alpha}\colon\alpha\in\fqn\}$, which is a field isomorphic to $\fqn$.
    \item $\mathcal{F}_{n,q}=\{\omega_{\alpha}\colon \alpha\in\fq\}$, which is a subfield of $\mathcal{F}_n$ isomorphic to $\fq$.
    \item $\tau_{\mathbf{v}}:\fqn\to W_1$, $\lambda\mapsto \lambda \mathbf{v}$, for any $\mathbf{v}\in W_1$.
\end{itemize}
We define a family of $\fq$-linear RM codes associated with an $\fq$-vector space $U$.

Let $n,r,k$ be positive integers with $k<rn$, $U$ be a $k$-dimensional $\fq$-subspace of an $r$-dimensional $\fqn$-vector space $V$, $W$ be an $(rn-k)$-dimensional $\fq$-vector space, and $G:V\to W$ be an $\fq$-linear map with kernel $U$.
For any $\mathbf{v} \in V$ define the $\fq$-linear map $\Gamma_{\mathbf{v}}=G\circ\tau_{\mathbf{v}}$.

\begin{theorem}\label{th:construction}
Let $V=V(r,q^n)$ and $W=V(rn-k,q)$. Let $U=V(k,q)$ be an $\fq$-subspace of $V$, and $G:V\to W$ be an $\fq$-linear map with $\ker(G)=U$.
Define
\[
\iota=\max\{\dim_{\fq}(U\cap\langle\mathbf{v}\rangle_{\fqn}) \colon \mathbf{v}\in V^*  \}.
\]
If $\iota<n$, then the pair $(U,G)$ defines an $\fq$-linear RM code
\begin{equation}\label{eq:rd}
    \mathcal{C}_{U,G}=\left\{ \Gamma_{\mathbf{v}}=G\circ \tau_{\mathbf{v}} \colon \mathbf{v} \in V \right\}
\end{equation}
of dimension $rn$ with parameters $(rn-k,n,q;n-\iota)$, whose right idealiser contains $\mathcal{F}_n$.
\end{theorem}

\begin{proof}
For any $\mathbf{v},\mathbf{w}\in V$ and $\alpha\in\fq$ we have $\Gamma_{\mathbf{v}}+\Gamma_{\mathbf{w}}=\Gamma_{\mathbf{v}+\mathbf{w}}$ and $\alpha\,\Gamma_{\mathbf{v}}=\Gamma_{\alpha\mathbf{v}}$, and hence $\mathcal{C}_{U,G}$ is an $\fq$-vector space.

For any $\mathbf{v}\in V$, let $R_{\mathbf{v}}=\{ \lambda \in \fqn \colon \lambda \mathbf{v} \in U \}$.
Clearly $\ker (\Gamma_{\mathbf{v}})=R_{\mathbf{v}}$ and, when $\mathbf{v}\ne\mathbf{0}$, $\dim_{\fq}(R_{\mathbf{v}})=\dim_{\fq}(U\cap\langle \mathbf{v}\rangle_{\fqn})$.
Then $\dim_{\fq}(\ker(\Gamma_{\mathbf{v}}))\leq\iota$ and there exists $\mathbf{u}\in V^*$ such that $\dim_{\fq}(\ker(\Gamma_{\mathbf{u}}))=\iota$, so that the minimum distance of $\mathcal{C}_{U,G}$ is $n-\iota$.

For any $\mathbf{v},\mathbf{w}\in V$, we have $\Gamma_{\mathbf{v}}=\Gamma_{\mathbf{w}}$ if and only if $\mathbf{v}=\mathbf{w}$.
In fact, if $\Gamma_{\mathbf{v}}=\Gamma_{\mathbf{w}}$, then $G(\lambda(\mathbf{v}-\mathbf{w}))=\mathbf{0}$ for every $\lambda\in\fqn$, whence $\dim_{\fq}(U\cap\langle\mathbf{v}-\mathbf{w}\rangle_{\fqn})=n>\iota$ and hence $\mathbf{v}=\mathbf{w}$.
Therefore, $\dim_{\fq}(\mathcal{C}_{U,G})=rn$.

Finally, for any $\alpha\in\fqn$ and $\mathbf{v}\in V$ we have $\Gamma_{\mathbf{v}}\circ \omega_{\alpha}=\Gamma_{\alpha\mathbf{v}}$. Then $R(\mathcal{C}_{U,G})$ contains $\mathcal{F}_n$.
\end{proof}

We now characterize the codes $\mathcal{C}_{U,G}$ which are MRD.

\begin{theorem}\label{th:MRDiff}
Let $V=V(r,q^n)$ and $W=V(rn-k,q)$. Let $U=V(k,q)$ be an $\fq$-subspace of $V$, $G:V\to W$ be an $\fq$-linear map with $\ker(G)=U$, $\iota=\max\{\dim_{\fq}(U\cap\langle\mathbf{v}\rangle_{\fqn}) \colon \mathbf{v}\in V^*  \}$ with $\iota<n$, and $\mathcal{C}_{U,G}=\left\{ \Gamma_{\mathbf{v}} \colon \mathbf{v} \in V \right\}$.

Then $\mathcal{C}_{U,G}$ is an $\fq$-linear MRD code if and only if
\[
(\iota+1)\mid rn\quad\textrm{and}\quad k=\frac{\iota rn}{\iota+1}\leq(r-1)n.
\]
In this case,
\begin{itemize}
    \item the parameters of $\mathcal{C}_{U,G}$ are $\left(\,\frac{rn}{\iota+1}\,,\,n\,,\,q\,;\,n-\iota\,\right)$;
    \item the right idealiser of $\mathcal{C}_{U,G}$ is $\mathcal{F}_n$;
    \item the weight distribution of $\mathcal{C}_{U,G}$ is
    \[ A_{n-s}={n \brack s}_q \sum_{j=0}^{\iota-s} (-1)^{j}{n-s \brack j}_q q^{\binom{j}{2}}\left(q^\frac{rn(\iota-s-j+1)}{\iota+1}-1\right), \]
    for $s \in \{0,1,\ldots,\iota\}$.
\end{itemize}
\end{theorem}

\begin{proof}
If $k>(r-1)n$, then for every $\mathbf{v}\in V^*$ we have $\dim_{\fq}(U\cap\langle\mathbf{v}\rangle_{\fqn})\geq1$ and hence $\dim_{\fq}(\ker(\Gamma_{\mathbf{v}}))=\dim_{\fq}(U\cap\langle\mathbf{v}\rangle_{\fqn})\geq1$, so that $\mathcal{C}_{U,G}$ has no elements of rank $n$. Thus, by Lemma \ref{lemma:weight}, $\mathcal{C}_{U,G}$ is not an MRD code.

Suppose $k\leq (r-1)n$. Then $rn-k\geq n$ and the Singleton-like bound of Theorem \ref{th:Singleton} reads
\[ rn\leq(rn-k)(n-(n-\iota)+1). \]
Therefore, $\mathcal{C}_{U,G}$ is an MRD code if and only if $\iota+1$ divides $rn$ and $k=\frac{\iota rn}{\iota+1}$.

In this case, the parameters of $\mathcal{C}_{U,G}$ are provided by Theorem \ref{th:construction} and the weight distribution of $\mathcal{C}_{U,G}$ follows from Theorem \ref{th:weightdistribution}.
Also, the right idealiser of $\mathcal{C}_{U,G}$ contains $\mathcal{F}_n$ by Theorem \ref{th:construction}, and hence is equal to $\mathcal{F}_n$ by Theorem \ref{th:propertiesideal}.
\end{proof}

Different choices of the map $G$ yield equivalent codes, i.e. $\mathcal{C}_{U,G}$ is uniquely determined by $U$, up to equivalence.

\begin{proposition}
Let $V=V(r,q^n)$ and $W=V(rn-k,q)$. Let $U=V(k,q)$ be an $\fq$-subspace of $V$, $G:V\to W$ and $\overline{G}:V\to W$ be two $\fq$-linear maps with $\ker(G)=\ker(\overline{G})=U$.
Then the codes $\mathcal{C}_{U,G}$ and $\mathcal{C}_{U,\overline{G}}$ are equivalent.
\end{proposition}

\begin{proof}
Let $B_U \cup\{\mathbf{w}_1,\ldots,\mathbf{w}_{rn-k}\}$ be an $\fq$-basis of $V$ such that $B_U$ is an $\fq$-basis of $U$.
Clearly, $G(\mathbf{w}_1),\ldots,G(\mathbf{w}_{rn-k})$ are $\fq$-linearly independent, as well as $\overline{G}(\mathbf{w}_1),\ldots,\overline{G}(\mathbf{w}_{rn-k})$.
Then there exists an invertible $\fq$-linear map $L:V\to V$ such that $L(U)=U$ and $L(G(\w_i))=\overline{G}(\w_i)$ for every $i=1,\ldots,rn-k$, i.e. $L\circ G=\overline{G}$.
Therefore, by choosing $R={\rm Id}_{\fqn}$ and $\sigma={\rm Id}_{{\rm Aut}(\fq)}$, we have
\[ L\circ\mathcal{C}_{U,G}^{\sigma}\circ R = \left\{L\circ(G\circ\tau_{\mathbf{v}})\colon \mathbf{v} \in V\right\} =\mathcal{C}_{U,\overline{G}}. \]
The claim is proved.
\end{proof}

We recall the following conjugacy property of Singer cycles of ${\rm GL}(n,q)$.

\begin{remark}\label{rem:singer}
The cyclic subgroups of $\mathrm{GL}(n,q)$ of order $q^n-1$ are called Singer cycles; it is well-known that any two Singer cycles $S_1=\langle g_1\rangle$ and $S_2=\langle g_2\rangle$ are conjugate in $\mathrm{GL}(n,q)$.

In fact, let $p_{g_1}(x)$ be the minimal polynomial of $g_1$ over $\mathbb{F}_q$, and $\gamma$ be a primitive element of $\fqn$ with minimal polynomial $p_{g_1}(x)$ over $\fq$.
The set $\overline{S}_1=S_1\cup\{\mathbf{0}\}$ is an $\fq$-subalgebra of $\fq^{n\times n}$, isomorphic to $\fqn$ by the $\fq$-linear map $\varphi$ mapping $(1,g_1,\ldots,g_1^{n-1})$ to $(1,\gamma,\ldots,\gamma^{n-1})$.
Also, $\overline{S}_1$ is a field of order $q^n$ and $\varphi$ is a field $\fq$-isomorphism.
The same holds for $\overline{S}_2=S_2\cup\{\mathbf{0}\}$, so that there exists a field $\fq$-isomorphism $\psi:\overline{S}_1\to\overline{S}_2$.
Therefore, there exists $\hat{\psi}\in\mathrm{GL}(n,q)$ which conjugates $S_1$ to $S_2$.
See also \cite[pag. 187]{Huppert} and \cite[Section 1.2.5 and Example 1.12]{Hiss}.
\end{remark}

Also the converse of Theorem \ref{th:MRDiff} holds, in the sense that any MRD code as in the claim of that theorem is equivalent to $\mathcal{C}_{U,G}$ for some $U$ as in the assumption of Theorem \ref{th:MRDiff}.

\begin{theorem}\label{th:MRDconverse}
Let $\mathcal{C}$ be an $\fq$-linear MRD code with parameters $(t,n,q;n-\iota)$ such that $t\geq n$ and $|R(\mathcal{C})|=q^n$, contained in ${\rm Hom}(\fqn,W)$ with $W=V(t,q)$.
Let $r=\dim_{R(\mathcal{C})}(\mathcal{C})$.
Then the following holds.
\begin{itemize}
    \item $\iota+1$ divides $rn$ and $t=\frac{rn}{\iota+1}$.
    \item $\mathcal{C}$ is equivalent to an $\fq$-linear MRD code $\mathcal{C}^\prime$ such that $R(\mathcal{C}^\prime)=\mathcal{F}_n$.
    \item The set \[ U=\{f\in\mathcal{C}^\prime\colon f(1)=0\}\subseteq\mathcal{C}^\prime \] is a $\frac{\iota rn}{\iota+1}$-dimensional $\mathcal{F}_{n,q}$-subspace of $\mathcal{C}^\prime$, and satisfies \footnote{Recall that $\langle f\rangle_{\mathcal{F}_n}=\{f\circ\omega_{\alpha}\colon\alpha\in\fqn\}$.}
    \begin{equation}\label{eq:iota} \max\left\{\dim_{\mathcal{F}_{n,q}}\left(U\cap\langle f\rangle_{\mathcal{F}_n}\right)\colon f\in\mathcal{C}^\prime\right\}=\iota. \end{equation}
    \item $\mathcal{C}^\prime$ is equal to $\mathcal{C}_{U,G}$, where $G:\mathcal{C}^\prime\to W$, $f\mapsto f(1)$.
\end{itemize}
\end{theorem}

\begin{proof}
Since $|\mathcal{C}|=q^{rn}$ and $t\geq n$, the Singleton-like bound of Theorem \ref{th:Singleton} reads $rn\leq t(n-(n-\iota)+1)$. As $\mathcal{C}$ is MRD, this implies that $\iota+1$ divides $t$, and $t=\frac{rn}{\iota+1}$.

Since $R(\mathcal{C})\setminus\{\mathbf{0}\}$ and $\mathcal{F}_n\setminus\{\omega_0\}$ are Singer
cycles of $\mathrm{GL}(n,q)$, there exists by Remark \ref{rem:singer} an invertible $\mathbb{F}_q$-linear map $H\colon \mathbb{F}_{q^n}\rightarrow\mathbb{F}_{q^n}$ such that
$R(\mathcal{C})= H\circ\mathcal{F}_n\circ H^{-1}$. Thus, $\mathcal{C}^\prime = \mathcal{C}\circ H$.

Clearly, $U$ is an $\mathcal{F}_{n,q}$-subspace of $\mathcal{C}^\prime$. 
For every $i\in\{1,\ldots,\iota\}$, we determine the size of $U_i=\{ f \in U \colon \dim_{\F_q}(\ker f)=i \}$.
Let $g \in \mathcal{C}'$ be such that $\dim_{\F_q}(\ker g)=i$.
As $\dim_{\F_q}(\ker g)>0$, there exists $\alpha \in \F_{q^n}^*$ such that $g(\alpha)=0$, that is $g\circ \omega_\alpha (1)=0$. As $\C'$ is a right vector space over $\mathcal{F}_n$, it follows that $g\circ \omega_\alpha \in \C'$ and, in particular, $g\circ \omega_\alpha \in U_i$.
This implies that
\[ \{ f \circ \omega_\alpha \colon f \in U_i,\,\,\alpha \in \F_{q^n}^* \} \]
coincides with the set of all the elements in $\C'$ of rank $n-i$.
Also, for any $f \in U_i$ and $\alpha\in\fqn$, we have $f \circ \omega_\alpha \in U_i$ if and only if $\alpha \in \ker f$. Thus,
\[ A_{n-i}= \frac{|U_i| (q^n-1)}{q^i-1}. \]
By Lemma \ref{lemma:complete weight} $A_{n-\iota}\ne0$, and \eqref{eq:iota} follows. Furthermore,
\[ |U|=1+ A_{n-1}\frac{q-1}{q^n-1}+\ldots+A_{n-\iota}\frac{q^\iota -1}{q-1}, \]
i.e.
\[ (q^n-1)(|U|-1)= A_{n-1}(q-1)+\ldots+A_{n-\iota}(q^\iota -1). \]
By Theorem \ref{th:dualrelations} applied to $\mathcal{C}'$ with $\nu=1$, we get \[ A_{n-1}(q-1)+\ldots+A_{n-\iota}(q^\iota -1)=(q^n-1)(q^{\frac{\iota rn}{\iota+1}}-1),\]
whence $\dim_{\mathcal{F}_{n,q}}(U)=\frac{\iota rn}{\iota +1}$.

Finally, choosing $G:\mathcal{C}^\prime\to W$, $f\mapsto f(1)$ and recalling that $\tau_{f}\colon \F_{q^n}\rightarrow \C'$, $\alpha \mapsto f\circ \omega_{\alpha}$ for any $f \in \C^\prime$, we obtain $\C'=\C_{U,G}$. 
\end{proof}


\noindent Theorems \ref{th:MRDiff} and \ref{th:MRDconverse} provide a correspondence between:
\begin{itemize}
    \item $\fq$-subspaces $U=V(\frac{\iota rn}{\iota+1},q)$ of $V=V(r,q^n)$ such that $\iota=\max\{\dim_{\fq}(U\cap\langle\mathbf{v}\rangle_{\fqn})\colon \mathbf{v}\in V^*\}$; and
    \item $\fq$-linear MRD codes $\mathcal{C}$ with parameters $\left(\frac{rn}{\iota+1},n,q;n-\iota\right)$ and right idealiser isomorphic to $\fqn$.
\end{itemize}

\noindent When $W=\F_{q^{{nr}/{(\iota+1)}}}$ and $R(\mathcal{C})=\mathcal{F}_n$, Theorem \ref{th:MRDconverse} reads as follows.

\begin{corollary}
Let $\iota,r,n$ be positive integers such that $\iota<n$, $\,\iota<r$ and $(\iota+1)\mid rn$.
Let $f_1,\ldots,f_r\colon \fqn \rightarrow \F_{q^{{nr}/{(\iota+1)}}}$ be $\mathcal{F}_n$-linearly independent (on the right) $\fq$-linear maps.
Then the RM code 
\[\C_{f_1,\ldots,f_r}=\{f_1\circ \omega_{\alpha_1}+\ldots+f_r\circ \omega_{\alpha_r} \colon \alpha_1,\ldots,\alpha_r \in \fqn\}\]
is an MRD code if and only if \[ \dim_{\fq} (\ker(f_1\circ \omega_{\alpha_1}+\ldots+f_r\circ \omega_{\alpha_r})) \leq \iota \]
for every $\alpha_1,\ldots,\alpha_r \in \fqn$.
In this case, $\mathcal{C}_{f_1,\ldots,f_r}$ has parameters $\left(\frac{rn}{\iota+1},n,q;n-\iota\right)$ and $R(\C_{f_1,\ldots,f_r})=\mathcal{F}_n$. Also, the $\mathcal{F}_{n,q}$-subspace $U_{f_1,\ldots,f_r}$ of $C_{f_1,\ldots,f_r}$ given by
\[ U_{f_1\ldots,f_r}=\{ f_1\circ \omega_{\alpha_1}+\ldots+f_r\circ \omega_{\alpha_r} \in \C_{f_1,\ldots,f_r} \colon f_1(\alpha_1)+\ldots+f_r(\alpha_r)=0 \} \]
has dimension $\frac{\iota rn}{\iota +1}$, and $\iota=\max\{\dim_{\fq}(U_{f_1,\ldots,f_r}\cap\langle\mathbf{v}\rangle_{\fqn})\colon \mathbf{v}\in \C_{f_1,\ldots,f_r}^*\}$.
\end{corollary}

\begin{example}
Let $\iota,n,r$ be positive integers such that $\iota<n$ and $(\iota+1)\mid r$. Define $t=r/(\iota+1)$.
The code
\[ \mathcal{C}=\left\{x\in\F_{q^{nt}}\mapsto a_0 x+a_1 x^q+\ldots+a_{\iota}x^{q^\iota}\in\F_{q^{nt}}\;\colon\; a_0,\ldots,a_{\iota}\in\F_{q^{nt}}\right\} \]
is an MRD code with parameters $(nt,nt,q;nt-\iota)$, known as Gabidulin code; see Section \ref{sec:noGab} below.
Consider the code \[\mathcal{C}|_{\fqn}=\left\{f|_{\fqn}:\fqn\to\F_{q^{nt}}\colon f\in\mathcal{C}\right\}.\]
By Theorem \ref{th:punct}, $\mathcal{C}|_{\fqn}$ is an MRD code with parameters $(nt,n,q;n-\iota)$.
Also, $R(\mathcal{C}|_{\fqn})=\mathcal{F}_n$ and an $\mathcal{F}_n$-basis of $\mathcal{C}|_{\fqn}$ (seen as a right vector space) is
\[ \left\{ f_{j,i} \colon x\in\fqn\mapsto \xi^{i}x^{q^j}\in\F_{q^{nt}}\;\mid\; 0\leq i\leq t-1,\,0\leq j\leq \iota \right\}, \]
where $\{1,\xi,\ldots,\xi^{t-1}\}$ is an $\fqn$-basis of $\F_{q^{nt}}$.
Moreover, the set of the elements $f\in\mathcal{C}|_{\fqn}$ vanishing at $1$ is equal to
\[ U=\left\{x\in\fqn\mapsto -(a_1+\ldots+a_{\iota}) x+a_1 x^q+\ldots+a_{\iota}x^{q^\iota}\in\F_{q^{nt}}\;\colon\; a_1,\ldots,a_{\iota}\in\F_{q^{nt}} \right\}, \]
and $\iota=\max\left\{\dim_{\mathcal{F}_{n,q}}\left(U\cap\langle f\rangle_{\mathcal{F}_n}\right)\colon f\in\mathcal{C}|_{\fqn}^{\,*}\right\}$.
Let
\[B=\left(f_{j,i}\,\colon\, j=0,\ldots,\iota,\;i=0,\ldots,t-1\right).\]
The coordinates of a vector in $U$ with respect to $B$ are
\[ \left( -\sum_{k=1}^{\iota}a_{k,0}\;,\ldots,-\sum_{k=1}^{\iota}a_{k,0}\;,\;a_{1,0}^{q^{n-1}},\ldots,a_{1,t-1}^{q^{n-1}}\;,\;\ldots\ldots,\;a_{\iota,0}^{q^{n-\iota}},\ldots,a_{\iota,t-1}^{q^{n-\iota}} \right), \]
where $a_{k,i}\in \fqn$ are such that $a_k=\sum_{i=0}^{t-1}a_{k,i}\xi^{i}$. Denote by $\overline{U}$ the set of the coordinates of the vectors in $U$.
Let $\sigma^{\prime}:\mathbb{F}_{q^n}^{t(\iota+1)}\times \mathbb{F}_{q^n}^{t(\iota+1)}\to\mathbb{F}_q$, $(\mathbf{u},\mathbf{v})\mapsto\mathrm{Tr}_{q^n/q}(\langle\mathbf{u},\mathbf{v}\rangle)$, where $\langle\cdot,\cdot\rangle$ is the standard inner product. Then the vectors of $\overline{U}^{\perp^\prime}$ are
\[ (y_0,\ldots,y_{t-1},y_0^{q^{n-1}},\ldots,y_{t-1}^{q^{n-1}},\ldots \ldots, y_0^{q^{n-\iota}},\ldots,y_{t-1}^{q^{n-\iota}}), \]
where $y_0,\ldots,y_{t-1}\in\mathbb{F}_{q^n}$.
Note that $\overline{U}^{\perp^\prime}$ is the direct sum of $t$ copies of
\[ \{ (z,z^{q},\ldots,z^{q^\iota})\colon z\in\mathbb{F}_{q^n} \} \]
which is a $\iota$-scattered $\fq$-subspace of $\fqn^{\iota+1}$. 

Therefore, when $\iota+1$ divides $r$, the restriction to $\fqn$ of a Gabidulin code is associated with the direct sum of $t$ copies of a $\iota$-scattered $\fq$-subspace of $\fqn^{\iota+1}$ (which is a $\iota$-scattered $\fq$-subspace of $\fqn^r$); see Section \ref{sec:twocharact}.
\end{example}

\section{Previously known connections}\label{sec:uni}

In this section, we show that the connection between $\fq$-vector spaces and $\fq$-linear MRD codes established in Section \ref{sec:subMRD} generalizes those presented in \cite{Sheekey,ShVdV,Lunardon2017,CSMPZ2016}.

\subsection{Sheekey's connection}\label{sec:sheekey}

The first connection was pointed out by Sheekey in its seminal paper \cite{Sheekey}.
Let $U$ be an $\fq$-subspace of $\fqn\times\fqn$, so that \[U=U_{f_1,f_2}=\{(f_1(x),f_2(x))\colon x\in\fqn\}\]
for some $f_1(x),f_2(x)$ in $\mathcal{L}_{n,q}$. Consider the $\fq$-linear RM code
\[\mathcal{S}_{f_1,f_2}=\{a_1 f_1(x)+a_2 f_2(x)\colon a_1,a_2\in\fqn\}\subset\mathcal{L}_{n,q},  \]
whose left idealiser is isomorphic to $\fqn$.
Then $U_{f_1,f_2}$ is a maximum scattered $\fq$-subspace of $\fqn\times\fqn$ if and only if $\mathcal{S}_{f_1,f_2}$ is an MRD code with parameters $(n,n,q;n-1)$; see \cite[Section 5]{Sheekey}.

\subsection{A generalization to maximum $(r-1)$-scattered $\fq$-subspaces of $\fqn^r$}\label{sec:r-1}

Sheekey's connection was extended by Sheekey and Van de Voorde in \cite{ShVdV} as follows; see also \cite{Lunardon2017}.
Let $U$ be an $\fq$-subspace of $\fqn^r$, so that
\[ U=U_{f_1,\ldots,f_r}=\{(f_1(x),\ldots,f_r(x))\colon x\in\fqn\} \]
for some $f_1(x),\ldots,f_r(x)\in\mathcal{L}_{n,q}$, and consider the $\fq$-linear RM code
\begin{equation}\label{eq:Cf1fr} \mathcal{S}_{f_1,\ldots,f_r}=\{a_1 f_1(x)+\cdots+a_r f_r(x)\colon a_1,\ldots,a_r\in\fqn\}, \end{equation}
whose left idealiser is isomorphic to $\fqn$.
Then $U_{f_1,\ldots,f_r}$ is a maximum $(r-1)$-scattered $\fq$-subspace of $\fqn^r$ if and only if $\mathcal{S}_{f_1,\ldots,f_r}$ is an MRD code with parameters $(n,n,q;n-r+1)$; see \cite[Corollary 5.7]{ShVdV}.
Clearly, when $r=2$ this connection coincides with the one of Section \ref{sec:sheekey}.

\subsection{A generalization to maximum scattered $\fq$-subspaces}\label{sec:JACO}

Sheekey's connection was extended by Csajb\'ok, Marino, Polverino and the last author in \cite{CSMPZ2016} by considering maximum scattered $\fq$-subspaces of $V=V(r,q^n)$ for any $r\geq2$ with $rn$ even; see \cite[Theorem 3.2]{CSMPZ2016}.

Let $U=V(\frac{rn}{2},q)$ be an $\fq$-subspace of $V$, $W=V(\frac{rn}{2},q)$, $G:V\to W$ be an $\fq$-linear map with $\ker(G)=U$, and $\iota=\max\{\dim_{\fq}(U\cap\langle\mathbf{v}\rangle_{\fqn})\colon \mathbf{v}\in V^*\}$ with $\iota<n$.
Then $\mathcal{C}_{U,G}=\{\Gamma_{\mathbf{v}}\colon \mathbf{v}\in V\}$ is an $\fq$-linear RM code of dimension $rn$ with parameters $(\frac{rn}{2},n,q;n-\iota)$.
Moreover, $U$ is a maximum scattered $\fq$-subspace of $V$ if and only if $\mathcal{C}_{U,G}$ is an MRD code.
In this case, the right idealiser of $\mathcal{C}_{U,G}$ is isomorphic to $\fqn$.

Conversely, in \cite{PZ} the authors prove that any $\fq$-linear MRD code with parameters $(\frac{rn}{2},n,q;n-1)$ and right idealiser isomorphic to $\fqn$ is equivalent to an MRD code $\mathcal{C}^\prime$ containing a maximum scattered $\fq$-subspace $U$ such that $\mathcal{C}^\prime=\mathcal{C}_{U,G}$ with $G:\mathcal{C}^\prime\to W$, $f\mapsto f(1)$; see \cite[Theorem 4.7]{PZ}.

This family contains the adjoint codes of the codes $\mathcal{S}_{f_1,f_2}$ presented in Section \ref{sec:sheekey}; see \cite[Example 3.5]{CSMPZ2016}.


\subsection{A unified connection}

When $\iota=1$ and $U$ is a maximum scattered $\fq$-subspace of $V=V(r,q^n)$, the connection established in Theorems \ref{th:construction} and \ref{th:MRDiff} coincides with the one of  Section \ref{sec:JACO}, and hence generalizes the one of Section \ref{sec:sheekey}.
Also, Theorem \ref{th:MRDconverse} extends the result of \cite{PZ}.

When $\iota=r-1$ and $U$ is a maximum $(r-1)$-scattered $\fq$-subspace of $V=V(r,q^n)$, our connection contains the adjoint codes of the MRD codes provided in Section \ref{sec:r-1}.
Indeed, let $\mathcal{C}$ be as in Equation \eqref{eq:Cf1fr} with parameters $(n,n,q;n-r+1)$ and left idealiser isomorphic to $\fqn$.
By Theorem \ref{th:propertiesideal}, the adjoint code $\mathcal{C}^\top$ is MRD with parameters $(n,n,q;n-r+1)$ and right idealiser isomorphic to $\fqn$.
Thus, by Theorem \ref{th:MRDconverse}, $\mathcal{C}^\top$ is equivalent to $\mathcal{C}_{U,G}$ for some $U$ and $G$.

\section{Two characterizations of  $h$-scattered subspaces}\label{sec:twocharact}

The $\fq$-subspaces $U$ of $V=V(r,q^n)$ defining an MRD code with parameters $(\frac{rn}{h+1},n,q;n-h)$ and right idealiser isomorphic to $\fqn$ are exactly those of dimension $\frac{hrn}{h+1}$ such that $h=\max\{\dim_{\fq}(U\cap\langle\mathbf{v}\rangle_{\fqn})\colon\mathbf{v}\in V^*\}$.
Examples of such $U$'s are provided by the ordinary duals of $\frac{rn}{h+1}$-dimensional $h$-scattered $\fq$-subspaces of $V$, for which several constructions are known; see \cite{CsMPZ,NPZZ}.
We prove that, whenever $n\geq h+3$, such $U$'s are exactly the ordinary duals of $\frac{rn}{h+1}$-dimensional $h$-scattered $\fq$-subspaces of $V$.
To this aim we provide two characterizations of these objects, namely Corollaries \ref{cor:caratterizzazione} and \ref{cor:caratterizzazione2}, by means of ordinary and Delsarte dualities.

\begin{theorem}\label{th:condhyper}
Let $r,n,h,k$ be positive integers such that $n\geq h+3$ and $k>r$.
Let $U$ be a $k$-dimensional $\fq$-subspace of $V=V(r,q^n)$ such that
\begin{equation}\label{eq:intersection} \dim_{\fq}(H\cap U)\leq k-n+h 
\end{equation}
for every $(r-1)$-dimensional $\fqn$-subspace $H$ of $V$.
Let $\Gamma,\V,\perp,\perp_D$ be as in Section \ref{sec:Delsarteduality}.
Then $U^{\perp_D}$ is an  $(n-h-2)$-scattered subspace of $\V/\Gamma^\perp$.
\end{theorem} 

\begin{proof}
As noted in Section \ref{sec:Delsarteduality}, there exist $\V=V(k,q^n)$, an $\fqn$-subspace $\Gamma=V(k-r,q^n)$ of $\V$, and an $\fq$-subspace $W=V(k,q)$ of $\V$ such that $\V=V\oplus\Gamma$, $\langle W\rangle_{\fqn}=\V$, $W\cap\Gamma=\{\mathbf{0}\}$, and $U=\langle W,\Gamma\rangle_{\fq}\cap V$.

Let $\perp^\prime$ and $\perp$ be the orthogonal complement maps which act respectively on the $\fq$-subspaces of $W$ and on the $\fqn$-subspaces of $\V$, which are defined by non-degenerate reflexive sesquilinear forms $\beta^\prime :W\times W\to\fq$ and $\beta:\V\times\V\to\fqn$ respectively, such that $\beta$ coincides with $\beta^\prime$ on $W\times W$.

Since $n\geq h+3$, $k>r$ and \eqref{eq:intersection} holds, the Delsarte duality can be applied to $U$.
Then $U^{\perp_D}=W+\Gamma^{\perp}$ is a $k$-dimensional $\fq$-subspace of $\V /\Gamma^{\perp}$.

Suppose that there exists an $(n-h-2)$-dimensional subspace $M$ of $\V/\Gamma^\perp$ such that $\dim_{\fq}(M\cap U^{\perp_D})\geq n-h-1$. Write $M=N+\Gamma^\perp$, where $N$ is an $(n-h-2+r)$-dimensional subspace of $\V$ satisfying $\Gamma^\perp \subseteq N$, so that
\[ \dim_{\fq}(N\cap W)=\dim_{\fq}(M\cap U^{\perp_D})\geq n-h-1. \]
Let $S$ be an $(n-h-1)$-dimensional $\fq$-subspace of $N\cap W$.
As $S\subseteq W$, we have $\dim_{\fqn}(S^*)=\dim_{\fq}(S)$; see \cite[Lemma 1]{Lun99}.
Since $N$ contains both $S$ and $\Gamma^\perp$, we have $N^\perp\subseteq (S^*)^\perp \cap \Gamma$, whence
\[ \dim_{\fqn}((S^*)^\perp \cap \Gamma)\geq \dim_{\fqn}(N^\perp) = k-(n-h-2+r). \]
This implies that $\langle(S^*)^\perp,\Gamma\rangle_{\fqn}$ is contained in an $\fqn$-subspace $T$ of $\V$ of dimension $k-1$.

Let $\hat{T}=T\cap V$.
As $T$ contains $\Gamma$, we have $\dim_{\fqn}(\hat{T})=r-1$.
Using $U=\langle W,\Gamma\rangle_{\fq}\cap V$ and $T\cap\langle W,\Gamma\rangle_{\fq}=\langle\Gamma,T\cap W\rangle_{\fq}$, we obtain
\[ \dim_{\fq}(\hat{T}\cap U)=\dim_{\fq}(T\cap W). \]
As $S^{\perp^\prime}= W\cap (S^*)^\perp\subseteq W\cap T$ and $\dim_{\fq}(S^{\perp^{\prime}})=k-(n-h-1)$, we obtain
\[ \dim_{\fq}(\hat{T}\cap U)\geq k-n+h+1, \]
a contradiction to \eqref{eq:intersection}.
Therefore $U^{\perp_D}$ is an $(n-h-2)$-scattered $\fq$-subspace of $\V/\Gamma^\perp$.
\end{proof}

\noindent Note that Theorem \ref{th:condhyper} can also be obtained as a consequence of \cite[Theorem 3.5]{BCsMT}.
By Theorem \ref{th:condhyper}, the following characterization is obtained.

\begin{corollary}\label{cor:caratterizzazione}
Let $r,n,h$ be positive integers such that $h+1$ divides $rn$ and $n\geq h+3$.
Let $U$ be an $\frac{rn}{h+1}$-dimensional $\fq$-subspace of $V=V(r,q^n)$.
Then $U$ is an $\frac{rn}{h+1}$-dimensional $h$-scattered $\fq$-subspace of $V$ if and only if  
\begin{equation}\label{eq:intermax} \dim_{\fq}(H\cap U)\leq\frac{rn}{h+1}-n+h 
\end{equation}
for every $(r-1)$-dimensional $\fqn$-subspace $H$ of $V$.
\end{corollary}
\begin{proof}
Assume that \eqref{eq:intermax} holds.
By Theorem \ref{th:condhyper}, $U^{\perp_D}$ is a $\frac{rn}{h+1}$-dimensional $(n-h-2)$-scattered $\fq$-subspace in $\V/\Gamma^\perp$.
By Proposition \ref{prop:property}, $U$ is a $\frac{rn}{h+1}$-dimensional $h$-scattered $\fq$-subspace of $V$.
The converse follows from Theorem \ref{th:inter}.
\end{proof}


\begin{remark}\label{rem:nonallargarti}
If $n>2$, $k<\frac{rn}{h+1}$ and $h=1$, then there exist $k$-dimensional $1$-scattered $\fq$-subspaces of $V=V(r,q^n)$ such that \eqref{eq:intersection} does not hold.
Therefore, Corollary \ref{cor:caratterizzazione} cannot be extended to all $h$-scattered subspaces which are not $\frac{rn}{h+1}$-dimensional $h$-scattered subspaces.

Indeed, let $U^\prime$ be a scattered $k^\prime$-dimensional $\fq$-subspace such that $H=\langle U^\prime\rangle_{\fqn}$ is a $(r-1)$-dimensional $\fqn$-subspace of $V$, and $\mathbf{v}\in V\setminus H$.
Then $U=U^\prime\oplus\langle\mathbf{v}\rangle_{\fq}$ is a $1$-scattered $\fq$-subspace of $V$ of dimension $k=k^\prime+1$ such that $\dim_{\fq}(U\cap H)=k-1>k-n+1$, as $n>2$.
\end{remark}

By using Corollary \ref{cor:caratterizzazione}, a further characterization of $\frac{rn}{h+1}$-dimensional $h$-scattered subspaces is proved.

\begin{corollary}\label{cor:caratterizzazione2}
Let $r,n,h$ be positive integers such that $h+1$ divides $rn$ and $n\geq h+3$.
Let $U$ be an $\frac{rn}{h+1}$-dimensional $\fq$-subspace of $V=V(r,q^n)$.
Then $U$ is an $\frac{rn}{h+1}$-dimensional $h$-scattered $\fq$-subspace of $V$ if and only if $U^{\perp_O}$ satisfies
\begin{equation}\label{eq:intermaxpoint} \dim_{\fq}(\langle \mathbf{v}\rangle_{\F_{q^n}}\cap U^{\perp_O})\leq h 
\end{equation}
for every $\mathbf{v}\in V\setminus\{\mathbf{0}\}$.
\end{corollary}

\begin{proof}
If $U$ is an $\frac{rn}{h+1}$-dimensional $h$-scattered $\F_q$-subspace of $V$, then the assertion follows from Theorem \ref{th:inter} and Equation \eqref{eq:dualweight}.
Conversely, $\dim_{\fq}(U)=\frac{rn}{h+1}$ implies $\dim_{\F_q}(U^{\perp_O})=\frac{hrn}{h+1}$.
Together with the assumption \eqref{eq:intermaxpoint} and Equation \eqref{eq:dualweight}, this yields
\[ \dim_{\fq}(H\cap U)\leq \frac{rn}{h+1}-n+h 
\]
for every $(r-1)$-dimensional $\F_{q^n}$-subspace $H$ of $V$.
The claim now follows from Corollary \ref{cor:caratterizzazione}.
\end{proof}

\noindent Theorems \ref{th:MRDiff} and \ref{th:MRDconverse}, together with Corollary \ref{cor:caratterizzazione2}, provide a correspondence between the following objects, under the assumption $n\geq h+3$:
\begin{itemize}
    \item $\frac{rn}{h+1}$-dimensional $h$-scattered $\fq$-subspaces of $V(r,q^n)$; and
    \item $\fq$-linear MRD codes with parameters $\left(\frac{rn}{h+1},n,q;n-h\right)$ and right idealiser isomorphic to $\fqn$.
\end{itemize}

\section{MRD codes inequivalent to generalized (twisted) Gabidulin codes}\label{sec:noGab}

In this section we prove that the family of RM codes described in Section \ref{sec:subMRD} contains MRD codes which are not equivalent to punctured generalized Gabidulin codes nor to punctured generalized twisted Gabidulin codes.

Let $N,k,s$ be positive integers with $k<N$ and $\gcd(s,N)=1$.
The \emph{generalized Gabidulin code} $\mathcal{G}_{k,s}$ is defined as
\[ \mathcal{G}_{k,s}=\left\{ x\in\F_{q^N}\mapsto a_0 x+ a_1 x^{q^s}+\ldots+a_{k-1}x^{q^{s(k-1)}}\in\F_{q^N}\,\colon\, a_0,\ldots,a_{k-1}\in\F_{q^N} \right\} \]
and is an $\fq$-linear MRD code with parameters $(N,N,q;N-k+1)$.
The codes $\mathcal{G}_{k,s}$ were first introduced in \cite{Delsarte,Gabidulin} for $s=1$ and generalized in \cite{kshevetskiy_new_2005}.

Let $0\leq c<N$ and $\eta\in\F_{q^N}$ be such that $\eta^{(q^N-1)/(q-1)}\ne(-1)^{Nk}$.
The \emph{generalized twisted Gabidulin code} $\mathcal{H}_{k,s}(\eta,c)$ is defined as
\[ \mathcal{H}_{k,s}(\eta,c)=\left\{ x\in\F_{q^N}\mapsto a_0 x+ a_1 x^{q^s}+\ldots+a_{k-1}x^{q^{s(k-1)}}+a_0^{q^c}\eta x^{q^{sk}}\in\F_{q^N}\,\colon\, a_i\in\F_{q^N} \right\} \]
and is an $\fq$-linear MRD code with parameters $(N,N,q; N-k+1)$.
The codes $\mathcal{H}_{k,s}(\eta,c)$ were first introduced in \cite{Sheekey} and investigated in \cite{LTZ}.

As a consequence of \cite[Theorem 3.8]{TZ}, the left idealisers of punctured generalized
(twisted) Gabidulin codes satisfy the following property.

\begin{lemma}\label{lemma:TZ}
Let $g:\F_{q^N}\to\F_{q^M}$ be an $\fq$-linear map of rank $M\leq N$, and consider the punctured code $\mathcal{C}$, where either $\mathcal{C}=g\circ\mathcal{G}_{k,s}$ or $\mathcal{C}=g\circ\mathcal{H}_{k,s}(\eta,c)$.
If $M>k+1$ and $(M,k)\ne(4,2)$, then $|L(\mathcal{C})|=q^\ell$ where $\ell$ divides $N$.
\end{lemma}

\begin{remark}
In Theorem \ref{th:new} we investigate the equivalence issue between the codes $\mathcal{C}$ as in \eqref{eq:rd} having parameters $(M,N,q;d)$ with $M\geq N$, and punctured generalized (twisted) Gabidulin codes.
As the punctured $(M,N,q;d)$-codes $\mathcal{D}$ arising from Lemma \ref{lemma:TZ} satisfy $M\leq N$, we need to consider the adjoint code $\mathcal{D}^{\top}$ of $\mathcal{D}$, having parameters $(N,M,q;d)$.
In this sense, whenever $\mathcal{C}$ and $\mathcal{D}^\top$ are not equivalent, we will say that $\mathcal{C}$ is not equivalent to a punctured generalized (twisted) Gabidulin code.
\end{remark}

We show in Theorem \ref{th:new} that the condition $(h+1)\nmid r$ is sufficient for the MRD codes of Section \ref{sec:subMRD} to be inequivalent to punctured generalized (twisted) Gabidulin codes.
Afterwards, we provide examples.

\begin{theorem}\label{th:new}
Let $r,n,h$ be positive integers such that $(h+1)$ divides $rn$, $n\geq h+3$, and $(n,h)\ne(4,1)$.
Let $\mathcal{C}=\mathcal{C}_{U,G}$ be the MRD code with parameters $\left(\frac{rn}{h+1},n,q;n-h\right)$ defined in Theorem \ref{th:MRDiff}.
If $(h+1)$ does not divide $r$, then $\mathcal{C}$ is not equivalent to any punctured generalized Gabidulin code nor to any punctured generalized twisted Gabidulin code.
\end{theorem}

\begin{proof}
Suppose that $\mathcal{C}$ is equivalent to $\mathcal{D}$, where $\mathcal{D}$ is either $(g\circ \mathcal{G}_{k,s})^\top$ or $(g\circ\mathcal{H}_{k,s}(\eta,c))^\top$.
Then $k=h+1$, $N=\frac{rn}{h+1}$ and $M=n$.
Since $n>h+2$ and $(n,h)\ne(4,1)$, we can apply Lemma \ref{lemma:TZ} and Theorem \ref{th:propertiesideal} to get $|R(\mathcal{D})|=q^\ell$, where $\ell$ divides $\frac{rn}{h+1}$.
As $\mathcal{C}$ and $\mathcal{D}$ are equivalent, by Theorem \ref{th:propertiesideal} we have $|R(\mathcal{C})|=|R(\mathcal{D})|=q^\ell$. Then $n=\ell$ and hence $n\mid\frac{rn}{h+1}$, a contradiction to $(h+1)\nmid r$.
\end{proof}

\begin{example}
Let $n\geq 6$ be even and $r^\prime\geq 3$ be odd.
By \cite[Theorem 3.6]{CsMPZ}, there exist $\frac{rn}{2}$-dimensional $(n-3)$-scattered $\fq$-subspaces $U^\prime$ of $V=V\left(\frac{r^\prime(n-2)}2,q^n \right)$.
Let $h=n-3$ and $r=\frac{r^\prime(n-2)}2$, and consider $U=U^{\prime\perp_O} \subseteq V$. 
Note that $(h+1)\nmid r$. 
Choose $G$ as in Theorem \ref{th:construction}.
By Theorems \ref{th:MRDiff} and \ref{th:new}, the MRD code $\C_{U,G}$ with parameters $\left(\frac{r^\prime n}{2},n,q;3\right)$ is not equivalent to any punctured generalized Gabidulin code nor to any punctured generalized twisted Gabidulin code.
\end{example}

\begin{example}
Let $n\geq6$ be even and $r\geq3$ be odd.
Examples 3.11, 3.12 and 3.13 in \cite{CSMPZ2016} provide MRD codes with parameters $\left(\frac{rn}{2},n,q;n-1\right)$ which are not equivalent to any punctured generalized Gabidulin code nor to any punctured generalized twisted Gabidulin code.
\end{example}

\section{$h$-scattered linear sets: intersection with hyperplanes and codes with $h+1$ weights}\label{sec:h+1weights}

Let $V=V(r,q^n)$. A point set $L$ of $\Omega=\PG(V,\F_{q^n})\allowbreak=\PG(r-1,q^n)$ is an \emph{$\F_q$-linear set} of $\Omega$ of rank $k$ if it is defined by the non-zero vectors of a $k$-dimensional $\F_q$-subspace $U$ of $V$, i.e.
\[L=L_U:=\{\la {\bf u} \ra_{\mathbb{F}_{q^n}} \colon {\bf u}\in U^* \}\}.\]
We denote the rank of $L_U$ by $\mathrm{rk}(L_U)$.
Let $\mathcal{S}=\PG(S,\F_{q^n})$ be a subspace of $\Omega$ and $L_U$ be an $\F_q$-linear set of $\Omega$. Then $\mathcal{S} \cap L_U=L_{S\cap U}$. If $\dim_{\F_q} (S\cap U)=i$, i.e. if $\mathcal{S} \cap L_U=L_{S\cap U}$ has rank $i$, we say that $\mathcal{S}$ has \emph{weight} $i$ in $L_U$, and we write $w_{L_U}(\mathcal{S})=i$.
Note that $0 \leq w_{L_U}(\mathcal{S}) \leq \min\{{\rm rk}(L_U),n(\dim(\mathcal{S})+1)\}$.
In particular, a point $P$ belongs to an $\F_q$-linear set $L_U$ if and only if $w_{L_U}(P)\geq 1$.
If $U$ is a (maximum) scattered $\F_q$-subspace of $V$, then we say that $L_U$ is (maximum) scattered.
In this case,
\[
|L_U| = \theta_{k-1}=\frac{q^k-1}{q-1},
\]
where $k$ is the rank of $L_U$; equivalently, all of its points have weight one. 

If $U$ is a (maximum) $h$-scattered $\F_q$-subspace of $V$, $L_U$ is said to be a (maximum) $h$-scattered $\F_q$-linear set in $\Omega=\PG(V,\F_{q^n})$. Therefore, an $\F_q$-linear set $L_U$ of $\Omega$ is $h$-scattered if
\begin{itemize}
  \item $\langle L_U \rangle= \Omega$;
  \item for every $(h-1)$-subspace $\mathcal{S}$ of $\Omega$, we have \[ w_{L_{U}}(\mathcal{S}) \leq h.\]
\end{itemize}
When $h=r-1$ and $\dim_{\F_q}(U)=n$, we obtain the scattered linear sets with respect to the hyperplanes introduced in \cite{Lunardon2017} and in \cite{ShVdV}.

By Theorem \ref{th:inter}, if $L_U$ is a $h$-scattered $\F_q$-linear set of rank $\frac{rn}{h+1}$ in $\Omega$, then for every hyperplane $\mathcal{H}$ of $\Omega$ we have
\[ \frac{rn}{h+1}-n\leq w_{L_U}(\mathcal{H})\leq \frac{rn}{h+1}-n+h. \]
The following question arises: 

\begin{center}
for any $j \in \{\frac{rn}{h+1}-n,\ldots,\frac{rn}{h+1}-n+h\}$, how many hyperplanes of $\Omega$ have weight $j$ in $L_U$?  
\end{center}

The answer is known for $h=1$, $h=2$ and $h=r-1$; see \cite{BL2000,NZ,ShVdV}.
Theorem \ref{th:intersectionhyper} gives a complete answer for any admissible values of $h$, $r$ and $n$.

\begin{theorem}\label{th:intersectionhyper}
Let $L_U$ be a $h$-scattered $\F_q$-linear set of rank $\frac{rn}{h+1}$ in $\Omega=\PG(r-1,q^n)$.
For every $i\in\left\{0,\ldots,h\right\}$, the number of hyperplanes of weight $\frac{rn}{h+1}-n+i$ in $L_U$ is
\begin{equation}\label{eq:ti} t_i=\frac{1}{q^n-1}{n \brack i}_q \sum_{j=0}^{h-i} (-1)^{j}{n-i \brack j}_q q^{\binom{j}{2}}\left(q^\frac{rn(h-i-j+1)}{h+1}-1\right). \end{equation}
In particular, $t_i>0$ for every $i\in\left\{0,\ldots,h\right\}$.
\end{theorem}

\begin{proof}
Let $\sigma,\sigma^\prime,\perp,\perp^\prime$ be defined as in Section \ref{sec:classicalduality}, and $\mathcal{H}=\PG(H,\F_{q^n})$ be a hyperplane of $\Omega$ with weight $\frac{rn}{h+1}-n+i$ in $L_U$.
By Equation \eqref{eq:dualweight},
\[ \dim_{\fq}(U^{\perp^\prime}\cap H^\perp)=\dim_{\fq}(U\cap H) +rn-\frac{rn}{h+1}-(r-1)n = i\leq h. \]
Thus, the $\F_q$-linear set $L_{U^{\perp^\prime}}$ of $\Omega$ has rank $\frac{hrn}{h+1}$, the weight in $L_{U^{\perp^\prime}}$ of a point of $\Omega$ is at most $h$, and the number of points of $\Omega$ with weight $i$ in $L_{U^{\perp^\prime}}$ equals the number $t_i$ of hyperplanes with weight $\frac{rn}{h+1}-n+i$ in $L_U$, for every $i\in \{0,\ldots,h\}$.
Let $W=V(\frac{rn}{h+1},q)$ and $G\colon V \rightarrow W$ be an $\F_q$-linear map with $\ker(G)=U$. By Theorem \ref{th:MRDiff}, the code $\mathcal{C}_{U^{\perp^\prime},G}=\{\Gamma_{\mathbf{v}} \colon \mathbf{v}\in V\}$ is an MRD code. 

Note that, for every $i \in \{0,\ldots,h\}$, $t_i$ is equal to the number of maps in $\C_{U^{\perp^\prime},G}$ having rank $n-i$, divided by $q^n-1$. In fact, if $P=\langle \mathbf{v}\rangle_{\F_{q^n}}$ is a point of weight $i$ in $L_{U^{\perp^\prime}}$, then $\Gamma_{\lambda\mathbf{v}}$ has rank $n-i$ for every $\lambda\in\F_{q^n}^*$; conversely, if $\mathbf{v}\in V^*$ is such that $\Gamma_{\mathbf{v}}$ has rank $n-i$, then $\langle\mathbf{v}\rangle_{\F_{q^n}}$ has weight $i$ in $L_{U^{\perp^\prime}}$.

Thus, $t_i=A_{n-i}/(q^n-1)$. By Theorem \ref{th:MRDiff}, Equation \eqref{eq:ti} follows. As $\mathcal{C}_{U^{\perp^\prime},G}$ is an MRD code, $t_i>0$ by Lemma \ref{lemma:complete weight}.
\end{proof}

Under the assumptions of Theorem \ref{th:intersectionhyper}, the property of being scattered determines completely the intersection numbers w.r.t.\ the hyperplanes.

\begin{corollary}\label{cor:intersec}
Let $L_U$ be a $h$-scattered $\fq$-linear set of rank $\frac{rn}{h+1}$ in $\Omega=\PG(r-1,q^n)$.
For every hyperplane $\mathcal{H}$ of $\Omega$, we have
\[ |\mathcal{H}\cap L_U| \in \left\{ \theta_{\frac{rn}{h+1}-n-1},\ldots,\theta_{\frac{rn}{h+1}-n+h-1} \right\}. \]
For every $i\in\{0,\ldots,h\}$, the number of hyperplanes $\mathcal{H}$ of $\Omega$ satisfying $|\mathcal{H}\cap L_U|=\theta_{\frac{rn}{h+1}-n+i-1}$ is $t_i$, as in Equation \eqref{eq:ti}. 
\end{corollary}


We now consider $h$-scattered $\fq$-linear sets $L_U$ of rank $\frac{rn}{h+1}$ as projective systems in $\Omega$, and the related linear codes (with the Hamming metric). By means of Theorem \ref{th:intersectionhyper} we determine the weight distribution and the weight enumerator.

\begin{theorem}
Let $L_U$ be a $h$-scattered $\fq$-linear set of rank $\frac{rn}{h+1}$ in $\Omega=\PG(r-1,q^n)$, and $\mathcal{C}_{L_U}$ be the corresponding
linear code over $\F_{q^n}$, having length $N=\theta_{\frac{rn}{h+1}-1}$ and dimension $k=r$.

Then $\mathcal{C}_{L_U}$ has minimum distance $d=\theta_{\frac{rn}{h+1}-1}-\theta_{\frac{rn}{h+1}-n+h-1}$ and exactly $h+1$ weights, namely 
$w_i=\theta_{\frac{rn}{h+1}-1}-\theta_{\frac{rn}{h+1}-n+i-1}$ with $i=0,\ldots,h$.
The weight enumerator of $\mathcal{C}_{L_U}$ is
\[1+\sum_{i=0}^{h}A_{w_i}^H z^{w_i},\]
where $A_{w_i}^H=t_i$, as in Equation \eqref{eq:ti}.
\end{theorem}

\begin{proof}
From ${\rm rk}(L_U)=\frac{rn}{h+1}$ follows the length $N=\theta_{\frac{rn}{h+1}-1}$, and from $\langle L_U\rangle = \Omega$ follows the dimension $k=r$.
By Corollary \ref{cor:intersec} and Proposition \ref{prop:projsyst}, the minimum distance of $\mathcal{C}_{L_U}$ is $d=\theta_{\frac{rn}{h+1}-1}-\theta_{\frac{rn}{h+1}-n+h-1}$. Also, the weight distribution and the weight enumerator of $\mathcal{C}_{L_U}$ are as in the claim.
\end{proof}

\begin{remark}
Randrianarisoa in \cite{Ra} introduced the concept of $[N,k,d]$ $q$-system over $\fqn$ as an $N$-dimensional $\fq$-subspace $U$ of $\fqn^k$ such that $\langle U\rangle_{\fqn}=\fqn^k$ and $N-d=\max\left\{\dim_{\fq}(U\cap H)\colon H=V(k-1,q^n)\subset \fqn^k\right\}$.
For any positive integers $r,n,h$ such that $(h+1)\mid rn$ and $n\geq h+3$, Corollary \ref{cor:caratterizzazione} implies that the $[\frac{rn}{h+1},r,n-h]$ $q$-systems over $\fqn$ are exactly the $\frac{rn}{h+1}$-dimensional $h$-scattered $\fq$-subspaces of $\fqn^r$.
In this case, the code $\mathcal{C}$ considered in \cite[Section 3]{Ra} has a generator matrix $G$ whose columns form an $\fq$-basis of $U$.
The code $\mathcal{C}$ turns out to be obtained by $\mathcal{C}_{L_U}$ by deleting all but $\frac{rn}{h+1}$ positions (corresponding to an $\F_q$-basis of $U$).
Together with \cite[Theorem 2]{Ra}, this answers the question posed in \cite[Section 8]{Ra} about the correspondence between $h$-scattered linear sets and RM codes of this type.
\end{remark}

\section{Conclusions and open questions}\label{sec:open}

Several connections between between MRD codes and scattered $\fq$-subspaces (linear sets) have been introduced in the literature. 
In this paper we propose a unified approach which generalizes all of these connections.
To this aim, we give useful characterizations of $\frac{rn}{h+1}$-dimensional $h$-scattered subspaces.
This allows to use the known constructions of $\frac{rn}{h+1}$-dimensional $h$-scattered subspaces in order to define MRD codes, and conversely.
The family we construct is very large and contains some ''new'' MRD codes, in the sense that they cannot be obtained by puncturing generalized (twisted) Gabidulin codes; this property is in general quite difficult to establish.
We conclude the paper by determining the intersection numbers of $h$-scattered linear sets of rank $\frac{rn}{h+1}$ w.r.t.\ the hyperplanes and the weight distribution of the code obtained by regarding the linear set as a projective system.

Several remarkable problems remain open; we list some of them.
\begin{itemize}
    \item The main open problem about $\frac{rn}{h+1}$-dimensional $h$-scattered $\fq$-subspaces of $V(r,q^n)$ is their existence for every admissible values of $r$, $n$ and $h$. This would imply the existence of possibly new MRD codes. Conversely, constructions of MRD codes with parameters $(\frac{rn}{h+1},n,q;n-h)$ and right idealiser isomorphic to $\fqn$, when $(h+1)\nmid r$ and $h< n-3$, give new examples of $\frac{rn}{h+1}$-dimensional $h$-scattered subspaces.
    \item Corollary \ref{cor:caratterizzazione2} characterizes $\frac{rn}{h+1}$-dimensional $h$-scattered subspaces whenever $n\geq h+3$. Is this characterization true also for $n<h-3$?
    \item Are there other families of MRD codes which can be characterized in terms of $\fq$-subspaces defining linear sets with a special behaviour?
\end{itemize}

Giovanni Zini and Ferdinando Zullo\\
Dipartimento di Matematica e Fisica,\\
Universit\`a degli Studi della Campania ``Luigi Vanvitelli'',\\
I--\,81100 Caserta, Italy\\
{{\em \{giovanni.zini,ferdinando.zullo\}@unicampania.it}}
\end{document}